\documentclass[12pt,leqno]{amsart}
\usepackage{amsmath,amssymb,amsfonts}
\usepackage{eucal,enumerate}

\newtheorem{theorem}{Theorem}

\theoremstyle{definition}

\newtheorem{example}{Example}

\newcommand\bbF{\mathbb{F}}
\newcommand\bfF{\mathbf{F}}
\newcommand \fG{\mathfrak G}
\newcommand\cH{\mathcal{H}}
\newcommand\bw{\bar{w}}
\newcommand\pdot{{\cdot}}
\newcommand\eset{\varnothing}
\newcommand\full{^{(V)}}
\newcommand\setm{\smallsetminus}
\newcommand\inv{^{-1}}
\newcommand\rk{\operatorname{rk}}

\newcommand\bs{\bar{s}}
\newcommand\triv{\mathbf{1}}

\newcommand\case[1]{\textit{Case }#1.}

\allowdisplaybreaks

\begin{document}

\title[Whitney Numbers of Partial Dowling Lattices]{Whitney Numbers of Partial Dowling Lattices}
\author{Thomas Zaslavsky}
\address{Department of Mathematics and Statistics, Binghamton University (SUNY),  Binghamton, NY 13902-6000, U.S.A.}
\email{zaslav@math.binghamton.edu}
\date{\today}

\begin{abstract}
The Dowling lattice $Q_n(\fG)$, $\fG$ a finite group, generalizes the geometric lattice generated by all vectors, over a field, with at most two nonzero components.  Abstractly, it is a fundamental object in the classification of finite matroids.  Constructively, it is the frame matroid of a certain gain graph known as $\fG\pdot K_n^{(V)}$.  Its Whitney numbers of the first kind enter into several important formulas.  Ravagnani suggested and partially proved that these numbers of $Q_n(\fG)$ and higher-weight generalizations are polynomial functions of $|\fG|$.  We give a simple proof for $Q_n(\fG)$ and its generalization to a wider class of gain graphs and biased graphs, and we determine the degrees and coefficients of the polynomials.
\end{abstract}


\maketitle

The famous Dowling lattices $Q_n(\fG)$ of a group $\fG$ \cite{CGL}, or more precisely their matroids, are one of the two fundamental types of matroid in the classification of Kahn and Kung \cite{VCG}.  They generalize to arbitrary groups the matroid of all vectors with at most two nonzero components in the vector space $F^n$, $F$ a field.  
We show that their Whitney numbers of the first kind, which are the coefficients of their characteristic polynomials and consequently of some importance, are themselves polynomials in the order of the group.  This is obvious from Dowling's own formulas, but we notice that this property generalizes considerably, to what we might call partial Dowling lattices.

My attention was drawn to this property by Ravagani's study \cite{Rav} of the same property of the higher-weight Dowling lattices, which he calls $\cH(n,q,d)$.  Such a lattice is the lattice of flats of the matroid of all vectors in $\bbF_q^n$ with at most $d$ nonzero coordinates.  The examples of weight $d=1$ gives a free matroid; but those of weight $d=2$, which are the Dowling lattices $Q_n(\bbF_q^\times)$, are more complicated; and the higher-weight Dowling lattices have resisted analysis, in particular of their important characteristic polynomials which have implications for linear coding theory.  Ravagnani suggested that the coefficients of these polynomials, which were dubbed Whitney numbers of the first kind by G.-C.\ Rota, are polynomial functions of $q$ and obtained partial results.  Stimulated by this work I looked at a different generalization, still in the realm of weight 2, and obtained conclusions that are fundamentally simple but with some complicated details.

\section{Gains and Dowlings}

We begin by introducing Dowling's lattices $Q_n(\fG)$ of a group \cite{CGL} by way of their matroids.  

A \emph{gain graph} $\Phi=(\Gamma_\Phi,\phi)$ consists of a graph $\Gamma_\Phi = (V,E)$ and a \emph{gain mapping} $\phi$ from oriented edges to a group $\fG$.  That is, the gain $\phi(e)$ depends on the orientation of the edge so that reversing the orientation inverts the gain, which we express in a formula as $\phi(e\inv) = \phi(e)\inv$, where $e$ is an oriented edge and $e\inv$ denotes the opposite orientation.  A particular kind of gain graph is the $\fG$-expansion of a graph $\Gamma$, written $\fG\pdot\Gamma$, whose vertex set is $V=V(\Gamma)$ and which has an edge $ge$ with every possible gain $g\in\fG$ for each edge $e\in E(\Gamma)$.  A \emph{half edge} is an edge with only one endpoint; we may add to the $\fG$-expansion a half edge at each vertex of a subset $X \subseteq V$, giving the \emph{partially filled} $\fG$-expansion $\fG\pdot\Gamma^{(X)}$.  In particular, when $X=V$ we call the result the \emph{full} $\fG$-expansion of $\Gamma$.

The associated \emph{frame matroid} $\bfF(\Phi)$ of a gain graph $\Phi$ (see \cite[Section 2]{BG2}) is defined on the edge set $E(\Phi)$ with circuits of three kinds.  The gain of a circle (i.e., a ``cycle'' or ``polygon'') $C = e_1\cdots e_l$ is $\phi(C) = \phi(e_1)\cdots\phi(e_l)$; this is sufficiently well defined because we only care whether the gain is the neutral element or not.  If the gain is the neutral element, $C$ is called \emph{neutral}.  The three kinds of circuit are these:
\begin{enumerate}[\quad C1.]
\item A neutral circle.
\item Two disjoint non-neutral circles together with a minimal connecting path, or two distinct non-neutral circles with exactly one common vertex.  Either or both circles may be replaced by a half edge in this type of circuit.
\item A theta subgraph whose three circles are not neutral.
\end{enumerate}
A subgraph or edge set is \emph{balanced} if every circle is neutral and there are no half edges.  The rank of an edge set $S$ is 
$$\rk(S) = |V| - b(S),$$
where $b(S)$ denotes the number of balanced components of the spanning subgraph $(V,S)$.  An edge set is \emph{closed}, or a \emph{flat} of $\bfF(\Phi)$, if it is maximal with its rank, i.e., adding any edge increases the rank.

The usual cycle matroid of $\Gamma$ is the frame matroid of the gain graph consisting of the graph $\Gamma$ with identity gains from the trivial group.  We write $\bfF(\Gamma)$ for this matroid.

The \emph{Dowling matroid} of rank $n$ of a group $\fG$ is $\bfF(\fG K_n\full)$ and the \emph{Dowling lattice} $Q_n(\fG)$ is the lattice of flats of this matroid \cite[Examples 3.6, 5.6]{BG3}.  Dowling originally considered $Q_n(\fG)$ as the lattice of the matroid formed by the weight 1 and 2 vectors in $\bbF_q^n$, that is, $\cH(n,q,2)$, in which context the group $\fG$ is the multiplicative group $\bbF_q^\times$ of order $q-1$.  (Both $q$ and $q-1$ appear naturally in Ravagnani's formulas for Whitney numbers of higher-weight Dowling lattices.  Does this suggest that both the additive and multiplicative groups of the field play a role in determining the matroid?)

\section{Chromatics and characteristics}

A gain graph has a chromatic polynomial $\chi_\Phi(\lambda)$ (see \cite[Section 3]{BG3}), one of whose definitions is
$$
\chi_\Phi(\lambda) := \sum_{S \subseteq E} (-1)^{|S|} \lambda^{b(S)}.
$$
We state the essential formulas.  We assume $\Phi$ has order $n$ and write $\gamma = |\fG|$.  The chromatic polynomial is monic of degree $n$.  If $b(\Phi)=0$ (as is the case for group expansions when $\gamma>1$ or $X=V$), the chromatic polynomial equals the characteristic polynomial $p_{\bfF(\Phi)}(\lambda)$ of the matroid; in general, 
\begin{equation*}\label{schar}
\chi_\Phi(\lambda) = \lambda^{b(\Phi)} p_{\bfF(\Phi)}(\lambda)
\end{equation*}
by \cite[Theorem 5.1]{BG3}.  Thus, the two polynomials have the same coefficients aside from a possible displacement.
From the definition, obviously 
\begin{equation}\label{noedges}
\chi_\Phi(\lambda) = \lambda^n \text{ if } E(\Phi)=\eset.
\end{equation}
A useful chromatic formula is the product rule,
\begin{equation}\label{sproduct}
\chi_{\Phi_1 \cup \Phi_2}(\lambda) = \chi_{\Phi_1}(\lambda)\chi_{\Phi_2}(\lambda)
\end{equation}
if $\Phi_1$ and $\Phi_2$ are disjoint gain graphs.  The usual chromatic polynomial of a graph is the chromatic polynomial of the trivial group expansion: $\chi_\Gamma(\lambda) = \chi_{\triv\pdot\Gamma}(\lambda)$, where $\triv$ denotes the trivial group.

We need the basic formula for chromatic polynomials of full group expansions, which is a reduction to ordinary graphs \cite[Examples 3.6 and 6.6]{BG3}:
\begin{equation}\label{sfull}
\chi_{\fG\pdot\Gamma\full}(\lambda) = \gamma^n \chi_\Gamma\Big(\frac{\lambda-1}{\gamma}\Big).
\end{equation}
We note that this polynomial is independent of the group structure.  One of Dowling's main discoveries was (in slightly different language) this simple formula for $\Gamma=K_n$, in contrast to the generally hard problem of computing characteristic polynomials of matroids.  

Formula \eqref{sfull} and the fact that $\chi_{\fG\pdot\Gamma\full}(\lambda)$ is a polynomial of degree $n$ clearly imply that the coefficients are polynomials in $\gamma$.  We want more detailed information.

\section{Whitneys}

The characteristic polynomial of a matroid $M$ has the form
$$
p_M(\lambda) = \sum_{i=0}^{\rk M} w_i(M) \lambda^{\rk M - i},
$$
the coefficients $w_i(M)$ being the \emph{Whitney numbers of $M$ of the first kind}.  The chromatic polynomial of $\Phi$ therefore has the form
$$
\chi_\Phi(\lambda) = \sum_{i=0}^{n} w_i(\Phi) \lambda^{n - i},
$$
where $w_i(\Phi)$ is defined as $w_i(\bfF(\Phi))$; similarly, $w_i(\Gamma)$ is defined as $w_i(\bfF(\Gamma))$.  The Whitney numbers of a matroid are nonzero, alternate in sign, and begin with $w_0=1$ and $-w_1 = |E(M)|$ (the number of atoms of $M$) \cite{FCT}, so the same is true for those of a gain graph; but note that $w_i(\Phi)=0$ for $i>\rk\bfF(\Phi)=n-b(\Phi)$.
Similarly, the Whitney numbers $w_i(\Gamma)$ are nonzero for $0 \leq i \leq n-c$ and 0 otherwise, where $c$ denotes the number of components of $\Gamma$ and $n-c=\rk\bfF(\Gamma)$.

We simplify our work by changing to the \emph{signless chromatic polynomial}, $\bar\chi_\Phi(\lambda) := (-1)^n \chi_\Phi(-\lambda)$, and the \emph{signless Whitney numbers}, which are $\bw_i = (-1)^iw_i = |w_i|$ by the alternating sign property.  Thus,
 $$
\bar\chi_\Phi(\lambda) = \sum_{i=0}^{n-b(\Phi)} \bw_i(\Phi) \lambda^{n - i}.
$$
Equation \eqref{sfull} becomes
\begin{equation}\label{full}
\bar\chi_{\fG\pdot\Gamma\full}(\lambda) = \gamma^n \bar\chi_\Gamma\Big(\frac{\lambda+1}{\gamma}\Big).
\end{equation}
For brevity we say ``Whitney number'' for both signed and signless Whitney numbers; the notation will show which is meant.

In group expansions we have, for example,
$
\bw_0(\fG\pdot\Gamma^{(X)}) = 1$ and $\bw_1(\fG\pdot\Gamma^{(X)}) = |E(\Gamma)| \gamma + |X|,
$
which obviously are polynomial functions of $\gamma$.  
An exercise for the reader is to compute $\bw_2$ before reading Formula \eqref{wformula}.

\begin{theorem}\label{Tfull} 
Given a simple graph $\Gamma$ of order $n>0$, a finite group $\fG$ of order $\gamma>0$, and an integer $i = 0,1,\dots,n$, the signless Whitney number $\bw_i(\fG\pdot\Gamma\full)$ is a polynomial function of $\gamma$ of degree $\min(i,n-c)$, with positive coefficients, i.e.,
\begin{equation}\label{wformula}
\bw_i(\fG\pdot\Gamma\full) = \sum_{0\leq j \leq i}  \bw_j(\Gamma) \binom{n-j}{i-j}  \gamma^j.
\end{equation}
\end{theorem}

\begin{proof}
The proof from Equation \eqref{full} is easy.
\begin{align*}
\bar\chi_{\fG\pdot\Gamma\full}(\lambda) &= \gamma^n \bar\chi_\Gamma\Big(\frac{\lambda+1}{\gamma}\Big)
= \sum_j \bw_j(\Gamma) \gamma^j (\lambda+1)^{n-j}
\\&
= \sum_j \bw_j(\Gamma) \gamma^{j} \sum_k \binom{n-j}{k} \lambda^k
\\&
= {\sum\sum}_{0\leq k \leq n-j\leq n} \lambda^k \bw_j(\Gamma) \gamma^{j} \sum_k \binom{n-j}{k} 
\end{align*}
and substituting $k=n-i$,
\begin{equation}\label{chiformula}
\bar\chi_{\fG\pdot\Gamma\full}(\lambda) = \sum_{0 \leq i \leq n} \lambda^{n-i} \sum_{0\leq j \leq i} \bw_j(\Gamma) \gamma^{j} \binom{n-j}{i-j} ,
\end{equation}
which implies \eqref{wformula}.
Thus, the Whitney numbers are polynomials in $\gamma$ and, as promised, all coefficients of powers $\gamma^j$ are positive when $j\leq i$ and $\bw_j(\Gamma)\neq0$, i.e., $j\leq n-c$.  Thus, the degree is $\min(i,n-c)$.
\end{proof}

When $i \leq n-c$, the leading coefficient in \eqref{wformula} is $\bw_i(\Gamma)$, which is the number of no-broken-circuit sets of $i$ edges in $\Gamma$ \cite{WhLogical}.  
When $i > n-c$, the leading coefficient is that of $\gamma^{n-c}$, which is $\bw_{n-c}(\Gamma)\binom{c}{n-i}$.  The first factor is the number of maximal forests that contain no broken circuit.  The second factor is the number of ways to choose $n-i$ of the $c$ components of the forest.  As a combinatorial interpretation of the product this seems arbitrary.
Are there meaningful combinatorial interpretations of the coefficients of powers of $\gamma$?

\section{More expansions}

We generalize to group expansions $\fG\pdot\Gamma^{(X)}$ with $X \subseteq V$, i.e., for any set of half edges.  (These were studied in \cite[Examples 3.7 and 6.7]{BG3}.)
We say a vertex set is \emph{stable} if it contains no edges in $\Gamma$.  The number of stable sets of order $k$ is $\alpha_k(\Gamma)$.  
The complement of $X \subseteq V$ is $X^c$.  The number of isolated vertices of $\Gamma$ that are not in $X$ is $\zeta$.

\begin{theorem}\label{Tnotfull} 
For any set $X \subseteq V(\Gamma)$, the Whitney number $\bw_i(\fG\pdot\Gamma^{(X)})$ is a polynomial function of $\gamma$ of degree $\min(i,n-c)$ for each $i=0,1,\ldots,n-\zeta$.  Specifically,
\begin{equation}\label{whits}
\bw_i(\fG\pdot\Gamma^{(X)}) = \sum_{j=0}^{i} \gamma^j \sum_{k=0}^{i-j} (-1)^k \binom{n-j-k}{n-i} \sum_{\substack{Y \subseteq X^c:\, \text{stable}\\ |Y|=k}} \bw_j(\Gamma\setm Y)
\end{equation}  
if $i \leq n-\zeta$ and $\bw_i(\fG\pdot\Gamma^{(X)}) = 0$ if $i > n-\zeta$.  The leading term is $\bw_i(\Gamma) \gamma^i$ if $i\leq n-c$ and $\bw_{n-c}(\Gamma) \binom{c-\zeta}{i-(n-c)} \gamma^{n-c}$ if $n-c \leq i \leq n-\zeta$.
\end{theorem}

\begin{proof}
The crucial formula is 
\begin{align*}\label{expansion}
\bar\chi_{\fG\pdot\Gamma^{(X)}}(\lambda) 
= \sum_{Y \subseteq X^c: \text{ stable}} (-1)^{|Y|} \bar\chi_{\fG\pdot\Gamma\full\setm Y}(\lambda),
\end{align*}
from \cite[Theorem 6.1]{BG3}.
We extract the Whitney numbers via Equation \eqref{chiformula}:
\begin{align*}
\sum_{\substack{Y \subseteq X^c\\ \text{stable}}} (-1)^{|Y|} \bar\chi_{\fG\pdot\Gamma\full\setm Y}(\lambda)
&= \sum_{\substack{Y \subseteq X^c\\ \text{stable}}} (-1)^{|Y|} \sum_{k=0}^{n-|Y|}  \bw_k(\fG\pdot\Gamma\full\setm Y) \lambda^{n-|Y|-k}
\\&
= \sum_{\substack{Y \subseteq X^c\\ \text{stable}}} (-1)^{|Y|} \sum_{k=0}^{n-|Y|}  \sum_{j=0}^k  \bw_j(\Gamma\setm Y) \binom{n-|Y|-j}{k-j}  \gamma^j \lambda^{n-|Y|-k}
\intertext{and now substituting $i=|Y|+k$,} 
&= \sum_{i=0}^n \lambda^{n-i}  \sum_{\substack{Y \subseteq X^c\\ \text{stable}\\ |Y|\leq i}} (-1)^{|Y|} \sum_{j=0}^{i-|Y|} \bw_j(\Gamma\setm Y) \binom{n-j-|Y|}{n-i} \gamma^j.
\end{align*}
The chromatic polynomial is therefore
\begin{equation}\label{unfull}
\bar\chi_{\fG\pdot\Gamma^{(X)}}(\lambda) = \sum_{i=0}^n \lambda^{n-i}  \sum_{j=0}^{i} \gamma^j \sum_{\substack{Y \subseteq X^c:\, \text{stable}\\ |Y|\leq i-j}} (-1)^{|Y|} \binom{n-j-|Y|}{n-i} \bw_j(\Gamma\setm Y).
\end{equation}
As the coefficient of $\lambda^{n-i}$ is $\bw_i(\fG\pdot\Gamma^{(X)})$, we obtain \eqref{whits} by introducing $k:=|Y|\leq i-j$.
This quantity is a polynomial in $\gamma$ of degree at most $i$.  The term of degree $i$ arises only from $Y=\eset$ and its coefficient is $\bw_i(\Gamma)$, which is nonzero as long as $i\leq n-c$.  It remains to find the degree when $i>n-c$.

Let $Z$ be the set of isolated vertices that are not in $X$, so $\zeta:=|Z|$.  From the product formula \eqref{sproduct} along with \eqref{noedges} we find that
$\bar\chi_{\fG\pdot\Gamma^{(X)}}(\lambda) = \lambda^\zeta \bar\chi_{\fG\pdot\Gamma^{(X)}\setm Z}(\lambda)
\text{ and }
\bar\chi_\Gamma(\lambda) = \lambda^\zeta \bar\chi_{\Gamma\setm Z}(\lambda);$
we conclude that $\bw_i(\fG\pdot\Gamma^{(X)}) = 0$ if $i>n-\zeta$. 
(This corresponds to the fact that $\rk\bfF(\fG\pdot\Gamma^{(X)}) = n-\zeta$ if $\gamma>1$ and is at most that if $\gamma=1$; see \cite[Section 2]{BG2}.)

Thus, now we assume $n-c \leq i \leq n-\zeta$.  The highest-degree term in $\bw_i(\fG\pdot\Gamma^{(X)})$ has degree at most $n-c$ because $\bw_j(\Gamma\setm Y) = 0$ when $j>n-c$.  We examine the coefficient of $\gamma^{n-c}$, i.e.,
\begin{equation}\label{topcoeff0}
\sum_{\substack{Y \subseteq X^c, \text{ stable}\\ |Y|\leq i-(n-c)}} (-1)^{|Y|} \bw_{n-c}(\Gamma\setm Y) \binom{c-|Y|}{n-i}.
\end{equation}  
Let $c_Y$ denote the number of components of $\Gamma \setm Y$.  The rank of $\bfF(\Gamma \setm Y)$ is $n-|Y|-c_Y$, so if $n-|Y|-c_Y < n-c$, then $\bw_{n-c}(\Gamma\setm Y)=0$ and the term of  $Y$ in \eqref{topcoeff0} drops out.  We may therefore restrict $Y$ to satisfy $|Y|+c_Y \leq c$.  Restating this as $c(\Gamma\setm Y) \leq c-|Y|$ shows that deleting $Y$ destroys at least $|Y|$ components, which is possible only if $Y$ consists of isolated vertices in $\Gamma$.  
As $Y\subseteq X^c$, this means $Y \subseteq Z$; indeed, $Y$ is any subset of $Z$ of the right size.  It also means that $\bw_j(\Gamma\setm Y) = \bw_j(\Gamma)$ since the Whitney numbers do not take account of isolated vertices.

We digress to prove an identity using Vandermonde convolution, assuming $c \geq m \geq 0$:
\begin{equation}\label{identity}
\begin{aligned}
\sum_{k=0}^{m} (-1)^k \binom{\zeta}{k} \binom{c-k}{c-m} 
&=\sum_{k=0}^{m} (-1)^k \binom{\zeta}{k} \binom{c-k}{m-k} 
= (-1)^m \sum_{k=0}^{m} \binom{\zeta}{k} \binom{-(c-m+1)}{m-k} 
\\&= (-1)^m \binom{\zeta-c+m-1}{m-k} 
= \binom{c-\zeta}{m}.
\end{aligned}
\end{equation}

Formula \eqref{topcoeff0} becomes
\begin{equation}\label{topcoeff}
\begin{aligned}
\bw_{n-c}(\Gamma) \sum_{\substack{Y \subseteq Z: |Y|\leq i-(n-c)}} (-1)^{|Y|} \binom{c-|Y|}{n-i} 
&= \bw_{n-c}(\Gamma) \sum_{k=0}^{i-(n-c)} (-1)^k \binom{\zeta}{k} \binom{c-k}{n-i} .
\\&= \bw_{n-c}(\Gamma)\binom{c-\zeta}{i-(n-c)}\end{aligned}
\end{equation} 
by \eqref{identity} with $m=i-(n-c)$ (so $n-i = c-m$).  Note that $c-\zeta \geq i-(n-c)$ because we assumed $i \leq n-\zeta$; therefore this coefficient is positive.  
Thus, $\bw_i(\fG\pdot\Gamma^{(X)})$ is a polynomial in $\gamma$ of degree exactly $n-c$ when $n-c < i \leq n-\zeta$.
\end{proof}

\section{Examples}

\subsection{Lower Whitney numbers}\

Here are formulas for the lower Whitney numbers in terms of the structure of $\Gamma$, from \eqref{whits}.  
The complement of the simple graph $\Gamma$ is $\Gamma^c$.  The degree in $\Gamma$ of a vertex $v$ is $d_\Gamma(v)$.  The number of triangles in $\Gamma$ is $t(\Gamma)$; note that $\bw_2(\Gamma) = \binom{|E(\Gamma)|}{2} - t(\Gamma)$.
\begin{equation}\label{lower}
\begin{aligned}
\bw_0(\fG\pdot\Gamma\full) &= 1,
\\
\bw_1(\fG\pdot\Gamma\full) &= |E(\Gamma)| \gamma + |X|,
\\
\bw_2(\fG\pdot\Gamma\full) &= 
\bigg[ \binom{|E(\Gamma)|}{2} - t(\Gamma) \bigg] \gamma^2 
+ \bigg[ (|X|-1)|E(\Gamma)| + \sum_{y \in X^c}  d_\Gamma(y) \bigg] \gamma 
\\&\qquad 
+ \bigg[ (n-1) \Big(|X| - \frac{n}{2}\Big) + |E(\Gamma^c\setm X)| \bigg]
\end{aligned}
\end{equation}

\subsection{Special coefficients}\

The constant term of $\bw_i(\fG\pdot\Gamma^{(X)})$, which is the value of the Whitney number when $\gamma=0$, is 
$$
\sum_{k=0}^{i} (-1)^k \binom{n-k}{i-k} \alpha_k(\Gamma \setm X).
$$
Suppose $\Gamma=K_n$, $n>1$.  Then $\zeta=0$.  There are no stable sets larger than a single vertex.  The constant term is
$$
\sum_{k=0}^1 (-1)^k \binom{n-k}{i-k} \alpha_k(K_n \setm X) = \binom{n-1}{n-i}|X| - (i-1)\binom{n}{i}.
$$
Note that $\binom{n-1}{-1} = 0$ and that by setting $\gamma=0$ we do not get the chromatic polynomial of a partially filled edgeless graph, $(K_n^c)^{(X)}$.

We state the two highest terms in $\gamma$ of $\bw_i(\fG\pdot\Gamma^{(X)})$, assuming $i\leq n-c$.  The highest power, $\gamma^i$, has coefficient $\bw_j(\Gamma)$.  The next lower coefficient is more interesting: it is
$$
(n+1-i) \bw_{i-1}(\Gamma) - \sum_{y \notin X} \bw_{i-1}(\Gamma\setm y).
$$

\subsection{A short path}

As a thorough example we develop the formulas for $\Gamma=P_2 = v_1v_2v_3$, the path of length 2.  We deduce the Whitney numbers from Equation \eqref{whits}.  The stable sets are $\eset$, $\{v_p\}$ for $p=1,2,3$, and $\{v_1,v_3\}$.  The Whitney numbers of $P_2$ are $\bw_0(P_2)=1$, $\bw_1(P_2)=2$, $\bw_2(P_2)=1$, $\bw_3(P_2)=0$.
\begin{align*}
\bw_i(\fG\pdot P_2^{(X)}) &
= \sum_{j=0}^{i} \gamma^j \sum_{k=0}^{i-j} (-1)^k \binom{3-j-k}{3-i} \sum_{\substack{Y \subseteq X^c:\, \text{stable}\\ |Y|=k}} \bw_j(P_2\setm Y).
\end{align*}  
We conclude that
\begin{equation}\label{P2}
\begin{aligned}
\bw_i(\fG\pdot P_2^{(X)}) 
&= \sum_{k=0}^2 (-1)^k \binom{3-k}{3-i} \alpha_k(P_2 \setm X)
\\&\quad
+ \gamma \sum_{k=0}^2 (-1)^k \binom{3-1-k}{3-i} \sum_{\substack{Y \subseteq X^c:\, \text{stable}\\ |Y|=k}} \bw_1(P_2\setm Y)
\\&\quad
+ \gamma^2 \sum_{k=0}^2 (-1)^k \binom{3-2-k}{3-i} \sum_{\substack{Y \subseteq X^c:\, \text{stable}\\ |Y|=k}} \bw_2(P_2\setm Y).
\end{aligned}
\end{equation}

There are six different possible sets $X$ up to isomorphism:
$$X = \eset, \{v_1\}, \{v_2\}, \{v_1,v_2\}, \{v_1,v_3\}, V.$$
We present the Whitney numbers in Table \ref{table:P2} and abbreviated calculations for two of the six cases.
\begin{table}[htp]
\begin{center}
\begin{tabular}{|c|c|r|r|r||c|c|}
\hline
$X$	&$\bw_0$	&$\bw_1$	&$\bw_2$	&$\bw_3$	&$\gamma=1$	&$\bfF(\triv\pdot P_2^{(X)})$
\\
	\hline
$\eset $	&$1$	&$2\gamma$	&$-2 + 2\gamma + \gamma^2$	&$-1 + \gamma^2$	&1, 2, 1, 0	&$\bfF(C_4)$	
\\
$\{v_1\}$	&$1$	&$1 + 2 \gamma$	&$-1 + 3 \gamma + \gamma^2$	&$-1 + \gamma + \gamma^2$	&1, 3, 3, 1	&$F_3$	
\\
$\{v_2\}$	&$1$	&$1 + 2 \gamma$	&$2 \gamma + \gamma^2$	&$\gamma^2$	&1, 3, 3, 1	&$F_3$	
\\
$\{v_1,v_2\}$ &$1$	&$2 + 2 \gamma$	&$1 + 3 \gamma + \gamma^2$	&$\gamma + \gamma^2$	&1, 4, 5, 2	&$\bfF(C_3) \oplus \bfF(K_2)$	
\\
$\{v_1,v_3\}$ &$1$	&$2 + 2 \gamma$	&$1 + 4 \gamma + \gamma^2$	&$2 \gamma + \gamma^2$	&1, 4, 6, 3	&$\bfF(C_4)$	
\\
$V$	&$1$	&$3 + 2 \gamma$	&$3 + 4 \gamma + \gamma^2$	&$1 + 2 \gamma + \gamma^2$	&1, 5, 8, 4	&$\bfF(K_4\setm e)$	
\\
\hline
\end{tabular}
\end{center}
\medskip
\caption{The Whitney numbers of the six expansions $\fG\pdot P_2^{(X)}$.  $F_3$ is the free matroid.  The fact that the formulas with $\gamma=1$ give the right Whitney numbers for the matroid of $\triv\pdot P_2^{(X)}$ in all six cases gives confidence that the formulas are correct.}
\label{table:P2}
\end{table}%

\case{2}  Consider $X=\{v_1\}$, so $X^c = \{v_2,v_3\}$ and $P_2 \setm X = K_2$.
\begin{align*}
\bw_i(\fG\pdot P_2^{(v_1)}) 
&= \sum_{k=0}^2 (-1)^k \binom{3-k}{3-i} \alpha_k(P_2 \setm v_1)
\\&\quad
+ \gamma \sum_{k=0}^2 (-1)^k \binom{3-1-k}{3-i} \sum_{\substack{Y \subseteq \{v_2,v_3\}:\, \text{stable}\\ |Y|=k}} \bw_1(P_2\setm Y)
\\&\quad
+ \gamma^2 \sum_{k=0}^2 (-1)^k \binom{3-2-k}{3-i} \sum_{\substack{Y \subseteq \{v_2,v_3\}:\, \text{stable}\\ |Y|=k}} \bw_2(P_2\setm Y)
\\
&= \bigg[ \binom{3}{3-i} - \binom{2}{3-i} 2  \bigg]
+ \gamma \bigg[ \binom{2}{3-i} 2 - \binom{1}{3-i} \bigg]
+ \gamma^2 \bigg[ \binom{1}{3-i} \bigg].
\end{align*}  

\case{6}  Here $X = V$ so $X^c = \eset$.  We know the chromatic polynomial from \eqref{sfull}.  This enables us to check the formula \eqref{P2} for arbitrary values of $\gamma$.
\begin{align*}
\bw_i(\fG\pdot P_2\full) 
&= \bigg[ \binom{3}{3-i} \bigg]
+ \gamma \bigg[ \binom{2}{3-i} 2 \bigg] 
+ \gamma^2 \bigg[ \binom{1}{3-i} \bigg].
\end{align*}  
The Whitney numbers (Table \ref{table:P2}) should agree with 
\begin{align*}
\chi_{\fG\pdot K_3\full}(\lambda) &= \gamma^3 \frac{\lambda-1}{\gamma} \Big( \frac{\lambda-1}{\gamma} - 1 \Big)^2 
= (\lambda-1)(\lambda-[1+\gamma])^2 
\\&= \lambda^3 - \lambda^2 \Big[ 3+2\gamma \Big] + \lambda \Big[ 1+2\gamma+\gamma^2+2(1+\gamma) \Big] - \Big[ 1+2\gamma+\gamma^2 \Big]. 
\end{align*}  
from \eqref{sfull}, and they do.

\subsection{Dowling Examples}

\begin{example}[Dowlings]\label{X:dowl}
Equation \eqref{full} immediately gives the Whitney numbers of the Dowling lattices.  The base graph chromatic polynomial is $\chi_{K_n}(\lambda) = \sum_{j=0}^n s(n,n-j)\lambda^{n-j}$, where $s(n,n-j)$ is the Stirling number of the first kind.  Its sign is $(-1)^{n-j}$ if  $j<n$; $s(n,0)$ is 0 if $n>0$.  Write $\bs(n,n-j) = (-1)^{n-j} s(n,n-j)$ for the signless Stirling number.  Thus, $\bw_j(K_n)=\bs(n,n-j)$ and for the Dowling lattice we have the polynomial formula
$$
\bw_i(Q_n(\fG)) = \bw_i(\fG\pdot K_n\full) 
= \sum_{0\leq j \leq i}  \bs(n,n-j) \binom{n-j}{i-j}  \gamma^j.
$$
\end{example}

\begin{example}[More or Less Jointless Dowlings]\label{X:jtlessdowl}
These are the lattices of $\bfF(\fG\pdot K_n^{(X)})$ for any vertex set $X$, or in Kung's terminology any set of joints.  Kung introduced the term \emph{joint} for the half edges in a group expansion, when viewed as elements of a standard basis for the matroid \cite{VCG}.  If all joints are missing, he calls the lattice \emph{jointless}.  

For $K_n$ with $n>0$, we have $c=1$ and the only stable sets $Y$ are $\eset$ and $\{y\}$ for $y\in X^c$.
Equation \eqref{whits} becomes
\begin{align*}
\bw_i(\fG\pdot K_n^{(X)}) = &\sum_{j=0}^{i} \left[ \bs(n,n-j) \binom{n-j}{i-j} - |X^c| \bs(n-1,n-j-1) \binom{n-j-1}{i-j-1} \right] \gamma^j.
\end{align*}

In particular, for $\gamma=1$ this should give $\bs(n,n-i) = \bw_i(K_n)$.  Thus, we have an identity involving Stirling numbers:
\begin{align*}
\sum_{j=0}^{i} \left[ \bs(n,n-j) \binom{n-j}{i-j} - n \bs(n-1,n-j-1) \binom{n-j-1}{i-j-1} \right] = \bs(n,n-i).
\end{align*}
\end{example}

\section{Biased expansions}

An $\gamma$-fold biased expansion of $\Gamma$ \cite[Example 3.8]{BG3}, where $\gamma$ is a positive integer, is a combinatorial abstraction and generalization of a group expansion.  It has the same numerical properties as a group expansion \cite[Example 6.8]{BG3} but does not require a group (although it is true that if $\Gamma$ is 3-connected, every biased expansion is a group expansion \cite{AMQ}).  That is, all the preceding results apply equally to biased expansions.



\begin{thebibliography}{9}

\bibitem{CGL} T.A.\ Dowling, 
A class of geometric lattices based on finite groups. 
\emph{J.\ Combin.\ Theory Ser.\ B} {\bf 14} (1973), 61--86.  
Erratum. 
\emph{ibid}.\ {\bf 15} (1973), 211.  

\bibitem{VCG} Jeff Kahn and Joseph P.S.\ Kung,
Varieties of combinatorial geometries.  
\emph{Trans.\ Amer.\ Math.\ Soc.}\ {\bf 271} (1982), 485--499.  

\bibitem{Rav} Alberto Ravagnani, 
Whitney numbers of combinatorial geometries and higher-weight Dowling lattices.  
\emph{SIAM J.\ Appl.\ Algebra Geom.}\ 6 (2022), no.\ 2, 156--189.

\bibitem{FCT} Gian-Carlo Rota,
On the foundations of combinatorial theory:  I.\ Theory of M{\"o}bius functions.
\emph{Z.\ Wahrsch.\ verw.\ Gebiete} {\bf 2} (1964), 340--368.

\bibitem{WhLogical} H.\ Whitney,
A logical expansion in mathematics. 
\emph{Bull.\ Amer.\ Math.\ Soc.}\ {\bf 38} (1932), 572--579.

\bibitem{BG2} Thomas Zaslavsky,
Biased graphs.  II.\  The three matroids.  
\emph{J.\ Combin.\ Theory Ser.\ B} {\bf 51} (1991), 46--72.  

\bibitem{BG3} Thomas Zaslavsky,
Biased graphs.  III.\  Chromatic and dichromatic invariants.  
\emph{J.\ Combin.\ Theory Ser.\ B} {\bf 64} (1995), 17--88.

\bibitem{AMQ} Thomas Zaslavsky, 
Associativity in multiary quasigroups:  The way of biased expansions.
\emph{Aequat.\ Math.}\ {\bf 83} (2012), no.\ 1, 1--66.

\end{thebibliography}
\end{document}